\documentclass[11pt,reqno]{amsart}

\usepackage[utf8]{inputenc}
\usepackage[margin=1.25in]{geometry}
\usepackage{hyperref}
\usepackage{appendix}
\usepackage{amsfonts}
\usepackage{amsthm}
\usepackage{amssymb}
\usepackage{stmaryrd} 
\usepackage{amsmath}
\usepackage{amsthm}
\usepackage{mathrsfs}
\usepackage{lipsum} 
\usepackage{xcolor}

\theoremstyle{plain}
\newtheorem{theorem}{Theorem}[section]

\newtheorem{lemma}[theorem]{Lemma}
\newtheorem{Proposition}[theorem]{Proposition}

\newtheorem{Example}[theorem]{Example}
\newtheorem{Definition}[theorem]{Definition}
\newtheorem{Sampling}[theorem]{Sampling}

\theoremstyle{remark}

\newtheorem{remark}[theorem]{Remark}

\usepackage{biblatex} 
\numberwithin{equation}{section}
\title[Deformed Fréchet law in crossover regime]{Deformed Fréchet law for Wigner and sample covariance matrices with tail in crossover regime}

\author{Yi HAN}
\address{Department of Pure Mathematics and Mathematical Statistics, University of Cambridge.
}
\email{yh482@cam.ac.uk}
\thanks{Supported by EPSRC grant EP/W524141/1.}

\begin{document}

\begin{abstract}
Given $A_n:=\frac{1}{\sqrt{n}}(a_{ij})$ an $n\times n$ symmetric random matrix, with elements above the diagonal given by i.i.d. random variables having mean zero and unit variance. It is known that when $\lim_{x\to\infty}x^4\mathbb{P}(|a_{ij}|>x)=0$, then fluctuation of the largest eigenvalue of $A_n$ follows a Tracy-Widom distribution. When the law of $a_{ij}$ is regularly varying with index $\alpha\in(0,4)$, then the largest eigenvalue has a Fréchet distribution. An intermediate regime is recently uncovered in Diaconu \cite{diaconu2023more}: when $\lim_{x\to\infty}x^4\mathbb{P}(|a_{ij}|>x)=c\in(0,\infty)$, then the law of the largest eigenvalue converges to a deformed Fréchet distribution. In this work we vastly extend the scope where the latter distribution may arise. We show that the same deformed Fréchet distribution arises (1) for sparse Wigner matrices with an average of $n^{\Omega}(1)$ nonzero entries on each row; (2) for periodically banded Wigner matrices with bandwidth $p_n=n^{O(1)}$; and more generally for weighted adjacency matrices of any $k_n$-regular graphs with $k_n=n^{\Omega(1)}$. In all these cases, we further prove that the joint distribution of the finitely many largest eigenvalues of $A_n$ converge to a deformed Poisson process, and that eigenvectors of the outlying eigenvalues of $A_n$ are localized, implying a mobility edge phenomenon at the spectral edge $2$ for Wigner matrices. The sparser case with average degree $n^{o(1)}$ is also explored. Our technique extends to sample covariance matrices, proving for the first time that its largest eigenvalue still follows a deformed Fréchet distribution, assuming the matrix entries satisfy $\lim_{x\to\infty}x^4\mathbb{P}(|a_{ij}|>x)=c\in(0,\infty)$. The proof utilizes a universality result recently established by Brailovskaya and Van Handel \cite{brailovskaya2022universality}.

\end{abstract}

\maketitle

\section{Introduction}

Consider $A_n=\frac{1}{\sqrt{n}}(a_{ij})_{1\leq i,j\leq n}\in\mathbb{R}^{n\times n}$ a symmetric random matrix with i.i.d. entries on and above the diagonal, such that $\mathbb{E}[a_{11}]=0$ and $\mathbb{E}[|a_{11}|^2]=1$. Denote by $\lambda_1(A_n)$ the largest eigenvalue of $A_n$. Such random matrix ensembles have been a central topic of study in physics and mathematics since the seminal work of Wigner \cite{wigner1958distribution} where the semicircle law was derived. When $a_{11}$ has Gaussian distribution, the fluctuation of $\lambda_1(A_n)$ is governed by the Tracy-Widom distribution (see for instance \cite{tracy1994level}), and the same phenomenon has been uncovered when $a_{11}$ does not have Gaussian distribution, see for instance \cite{tao2011random}.

A significant stream of recent research concerns weakening the moment assumptions on $a_{ij}$, that is to find the minimal moment assumption on $a_{ij}$ such that $A_n$ still behaves like the Gaussian case, or has a fundamentally different behavior. Bai and Yin \cite{bai1988necessary} first showed that a finite fourth moment condition on $a_{ij}$ is sufficient for $\lambda_1(A_n)$ to stick almost surely to 2, the spectral edge. Meanwhile, when $\mathbb{P}(|a_{11}|>x)$ decays like $x^{-\alpha}$ for $\alpha\in(0,4),$ then the edge eigenvalues (i.e. the finitely many largest eigenvalues) of $A_n$ behave like a properly scaled Poisson point process \cite{Soshnikov2004}, \cite{article}.
The same transition has also been proved for random band matrices having $N^\mu$ elements in each row, where each $a_{ij}$ has tail $x^{-\alpha}$ \cite{benaych2014localization}: if $\alpha<2(1+\mu^{-1})$ the edge eigenvalues stick to the spectrum; while if $\alpha>2(1+\mu^{-1})$ the edge eigenvalues form a Poisson point process. Finally, on the Tracy-Widom scale, Lee and Yin \cite{Lee2012ANA} proved that a necessary and sufficient condition for the edge eigenvalues of a Wigner matrix to follow Tracy-Widom distribution is that $\lim_{x\to\infty}x^4\mathbb{P}(|a_{ij}|>x)=0.$

The aforementioned works have left out the intermediate case $\lim_{x\to\infty}x^4\mathbb{P}(|a_{ij}|>x)=c\in(0,\infty),$ and this case is recently addressed in Diaconu \cite{diaconu2023more}. Specifically, \cite{diaconu2023more} proved that in this intermediate case as $n\to\infty,$ 
\begin{equation}\label{mainlimit}
    \lambda_1(A_n) \overset{\text{law}}{\underset{n\to\infty}\longrightarrow} f(\xi_c),
\end{equation}
where \begin{equation}\label{equationoff}  f(x)=
\begin{cases} x+\frac{1}{x},\quad x\geq 1,\\ 2,\quad 0<x<1,\end{cases}
\end{equation}
and for any $x>0$, $\mathbb{P}(\xi_c\leq x)=\exp(-\frac{cx^{-4}}{2}).$

The proof of \cite{diaconu2023more} utilizes the method of moments and manipulates heavy combinatorial structures. In this paper we take a different perspective and illustrate the universality of \eqref{mainlimit} from the following five perspectives: (1) we characterize the joint distribution of the finitely many largest eigenvalues; (2) we consider sparsely diluted matrix $B_n\circ A_n$; (3) we consider random band matrices and weighted adjacency matrices of $k_n$-regular graphs, settling the intermediate regime left out in \cite{benaych2014localization}; (4) we study eigenvectors associated to these outlying eigenvalues and establish a mobility edge phenomenon; and (5) we show a similar distribution arises for sample covariance matrices with tail in this crossover regime, which settles the intermediate regime for the Wishart case in \cite{article}.

\subsection{Wigner matrices and sparse Wigner matrices}

The main result of this paper for (sparse) Wigner matrices is stated as follows. The special case $\mu=1$ covers the result in \cite{diaconu2023more}, and our proof is more conceptual.

\begin{Definition}\label{definition1.1}(Dense and sparse Wigner matrix)
Consider $B_n=\{b_{ij}\}_{1\leq i,j\leq n}\in\mathbb{R}^{n\times n}$ a symmetric random matrix, whose entries satisfy $b_{ij}=b_{ji}$ and $\operatorname{Law}(b_{ij})=\operatorname{Ber}(\frac{p_n}{n})$ a Bernoulli random variable with success probability $\frac{p_n}{n}$, for some given $0<p_n\leq n$.

We consider $\Xi_n$, a sparse matrix given by $A_n$ diluted by $B_n$, with expression $$\Xi_n:=\sqrt{\frac{n}{p_n}}B_n\circ A_n,$$

where $\circ$ denotes entry-wise product, and where $A_n=\frac{1}{\sqrt{n}}(a_{ij})_{1\leq i,j\leq n}$. That is, $$(\Xi_n)_{ij}=\frac{1}{\sqrt{p_n}}b_{ij}a_{ij},\quad 1\leq i,j\leq n.$$

\end{Definition}

\begin{theorem}\label{theorem1.2}
Consider the Winger matrix in Definition \ref{definition1.1}. Assume that for some $\mu\in(0,1]$, we have 
\begin{equation} p_n=n^\mu,\quad n\in\mathbb{N}_+.\end{equation}
Assume that the entries $(a_{ij})_{1\leq i\leq j\leq n}$ of $A_n$ are independent and identically distributed, take real value and have symmetric law, and satisfy $\mathbb{E}[a_{ij}]=0$, $\mathbb{E}[|a_{ij}|^2]=1$, and
\begin{equation}
   \lim_{x\to +\infty}x^{2(1+\mu^{-1})}\mathbb{P}(|a_{ij}|>x)=c\in(0,\infty).
\end{equation}
Let $\lambda_1(\Xi_n)\geq \lambda_2(\Xi_n)\geq\cdots\lambda_n(\Xi_n)$ denote the eigenvalues of $\Xi_n$, arranged in decreasing order. Then we have the following conclusion:

\begin{enumerate}
    \item (Top eigenvalue) We have the convergence in law
   \begin{equation}\label{topconvergence} \lambda_1(\Xi_n) \overset{\text{law}}{\underset{n\to\infty}\longrightarrow} f(\xi_c^\mu),\end{equation} where the function $f$ is defined in \eqref{equationoff}, and the distribution of $\xi_c^\mu$ satisfies
\begin{equation}
    \mathbb{P}(\xi_c^\mu\leq x)=\exp\left(-\frac{cx^{-2(1+\frac{1}{\mu})}}{2}\right),\quad x>0.
\end{equation}

    \item (Joint distribution) More generally, for any $k>0$, the top $k$ eigenvalues of $\Xi_n$ have the joint distribution 
    \begin{equation}\label{jointlaws}
        (\lambda_1(\Xi_n),\cdots,\lambda_k(\Xi_n)) \overset{\text{law}}{\underset{n\to\infty}\longrightarrow} (f(\zeta_1),\cdots,f(\zeta_k)),
    \end{equation}
    where $\zeta_1\geq \cdots\geq \zeta_k>0$ are the $k$-largest points (in order) sampled from a Poisson point process on $[0,+\infty)$ with intensity measure \begin{equation}   \frac{c(1+\frac{1}{\mu})}{x^{2(1+\frac{1}{\mu})+1}}dx,\quad x>0.\end{equation} The function $f$ is again defined by \eqref{equationoff}.
    \item (Eigenvector localization)  For each $k$ denote by $v_k=v_k(\Xi_n)$ the unit eigenvector of $\Xi_n$ corresponding to eigenvalue $\lambda_k=\lambda_k(\Xi_n)$. Consider, for any $\epsilon>0$, the event
    \begin{equation}\label{whatistheevents}
       \Omega_k^\epsilon:=\left( \begin{aligned}&\text{there exists  distinct indices $i,j\in[1,n]^2$  }\\&\text{and some $D=\pm 1$ such that } 
       \\& \left|\langle v_k,\frac{1}{\sqrt{2}}(\delta_i+D\delta_j)\rangle\right|\in (1-\frac{4}{(\lambda_k+\sqrt{\lambda_k^2-4})^2})(1+[-\epsilon,\epsilon])
        \end{aligned} \right), 
    \end{equation} where $\frac{1}{\sqrt{2}}(\delta_i+D\delta_j)$ is the unit-norm vector in $\mathbb{R}^n$ which is non-zero only at its $i$ and $j$ coordinates, taking values $\frac{1}{\sqrt{2}}$ and  $D\frac{1}{\sqrt{2}}$ respectively.

    Then for any $\epsilon>0$ and $k>0$, 
    \begin{equation}\label{finaljustification}
        \mathbb{P}\left(\Omega_k^\epsilon\mid \lambda_k>2\right){\underset{n\to\infty}\longrightarrow} 1. 
    \end{equation} 
    
    In particular, with probability $1-o(1)$ all eigenvectors associated to outlying eigenvalues have $L^\infty$ norm $\Omega(1)$, which means they are completely localized in the $L^\infty$ norm.
\end{enumerate}
\end{theorem}

\begin{remark}
   For the dense case $p_n=n$, part (3) of Theorem \ref{theorem1.2} establishes a mobility edge phenomenon for a Wigner matrix with tail decaying at $x^{-4}$. For such matrices the theorem proves that eigenvectors associated to outlying eigenvalues $\lambda>2$ are localized (i.e. a unit $L^2$-norm eigenvector has $L^\infty$ norm $\Omega(1)$), whereas by previous results of Aggarwal \cite{aggarwal2019bulk}, eigenvectors associated to bulk eigenvalues $\lambda\in(-2+\epsilon,2-\epsilon)$ are completely de-localized for any $\epsilon>0$ (i.e. the unit eigenvector has $L^\infty$ norm $O(n^{-1/2+\epsilon})$), so that $E=2$ is a mobility edge, the energy threshold separating the localized and de-localized phase. Recently, existence of a mobility edge is rigorously proved for adjacency matrices of Erdős-Rényi graphs on $n$ vertices with average degree $c\log n$ for $c$ in a certain regime \cite{alt2024localized}, and its mobility edge is also at the spectral edge $E=2$.

    Also, we can assume that the diagonal entries $(a_{ii})_{1\leq i\leq n}$ of $A_n$ are i.i.d. random variables having mean zero and a general variance $V\in(0,\infty)$ (not necessarily $V=1$), and are independent from the off-diagonal terms. The conclusion of Theorem \ref{theorem1.2} remains true with exactly the same proof.
\end{remark}

\subsubsection{Super-polynomial tails}
We may also consider sparser graphs $(\log n)^{\log\log n}<<p_n<<n^{O(1)}$. In this regime the result is slightly different:
\begin{theorem}\label{theorem1.2add}
Consider a function $g(x):\mathbb{R}_+\to\mathbb{R}_+$ which is 1-1 and monotonically increasing in $x$, and such that 
\begin{equation}\label{whatisg} \lim_{x\to\infty}\frac{g(x)}{\log\log x}=\infty,\quad  \lim_{x\to\infty} \frac{g(x)\log\log x}{\log x}= 0,    \end{equation}
and we make a regularity assumption
\begin{equation}\label{regularityg} \lim_{x\leq y,x,y\to\infty}\frac{\log(y)-\log(x)}{g(y)\log\log(y)-g(x)\log\log(x)}= \infty.    \end{equation}

 Consider the sparsity regime
\begin{equation} p_n=(\log n)^{g(n)},\quad n\in\mathbb{N}_+,\end{equation}
and suppose $h:\mathbb{R}_+\to\mathbb{R}_+$ is the unique function that satisfies, for some $a>1$:
\begin{equation}
    h\left(a\sqrt{\log(x)^{g(x)}}\right)=x,\quad x\in [1,\infty),
\end{equation}
and we assume, in addition to $\mathbb{E}[a_{ij}]=0$, $\mathbb{E}[|a_{ij}|^2]=1$,
\begin{equation}
    \lim_{x\to\infty} x^2h(x)\mathbb{P}(|a_{ij}|>x)=c\in(0,\infty).
\end{equation}
Then we have the convergence in distribution
   \begin{equation} \lambda_1(\Xi_n) \overset{\text{law}}{\underset{n\to\infty}\longrightarrow} f(a).\end{equation} 
\end{theorem}

\begin{remark} 
In this regime \eqref{whatisg}, instead of studying the distribution of $\lambda_1(\Xi_n)$, one may further study fluctuation of $\lambda_1(\Xi_n)$ around its deterministic limit $f(a)$. Our method can be adapted to prove that the law of fluctuation should be some modified Fréchet law, and the scale of fluctuation should be a function of $g(x)$. However, the fluctuation scale would be a very complicated expression involving $g$, so we omit it for simplicity.

We have also restricted ourselves to $p_n=\Omega((\log n)^{\log\log n})$, so that not too many elements in $\Xi_n$ would be large. One may also consider a sparser regime, say $p_n=(\log n)^\gamma$ for $\gamma>1$, and we expect a similar phenomenon when the tails of $a_{ij}$ are suitably defined. However, the techniques in this paper do not seem to work in the sparsity regime $p_n=(\log n)^\gamma$, and the problem does not seem well-motivated in that regime, so we restrict ourselves to the denser regime \eqref{whatisg}.
\end{remark}

\subsection{Weighted regular graphs}

The main result for periodically banded random matrices, and for weighted regular graphs, is stated as follows. For these matrix ensembles, \cite{article} discussed the effect of tails on the top eigenvalue: the top eigenvalue sticks to spectral edge when we have lighter tails; whereas the top eigenvalue will have a Fréchet distribution when heavier tails present. We address the crossover regime in this paper:

\begin{Definition}\label{banddefinition1.1}(Weighted regular graph)
Let $G_n=\{g_{ij}\}_{1\leq i,j\leq n}\in\mathbb{R}^{n\times n}$ be the adjacency matrix of a $k_n$-regular graph on vertices $\{1,2,\cdots,n\},$ allowing self loops. That is, for any vertex $i\in\{1,2,\cdots,n\}$, there are precisely $k_n$ vertices $j$ such that $g_{ij}=g_{ji}=1$, and we allow $j=i$. For all other vertices $l$ we have $g_{il}=g_{li}=0$. Assume that $k_n\in\mathbb{N}_+$ and $k_n\leq n$.

The weighted regular graph $\mathcal{M}_n$ generated from the Wigner matrix $A_n$ and the adjacency matrix $G_n$ is defined by $$\mathcal{M}_n:=\sqrt{\frac{n}{k_n}}G_n\circ A_n,$$
where $\circ$ denotes entry-wise product, and where $A_n=\frac{1}{\sqrt{n}}(a_{ij})_{1\leq i,j\leq n}$. That is, $$(\mathcal{M}_n)_{ij}=\frac{1}{\sqrt{k_n}}g_{ij}a_{ij},\quad 1\leq i,j\leq n.$$

\end{Definition}

Definition \ref{banddefinition1.1} includes as a special case the following important example:
\begin{Example}(Periodically banded matrices) For $i,j=1,2,\cdots,n$, define 
$$ |i-j|_n:=\min(|i-j|,n-|i-j|), 
$$ and consider $k_n=\min(2b_n+1,n)$ where $(b_n)_{n\in\mathbb{N}_+}\in\mathbb{N}_+$ are bandwidths. Then define 
$$
(G_n)_{ij}=1\{|i-j|_n\leq b_n\},
$$ so that $G_n$ is the adjacency matrix of a periodically banded matrix.
\end{Example}

Our main result concerning weighted regular graphs is as follows:

\begin{theorem}\label{theorem1.7}
Consider the weighted regular graph in Definition \ref{banddefinition1.1}. Assume that for some $\mu\in(0,1]$ and some $\eta>0$, we have 
\begin{equation} \frac{k_n}{n^\mu}=1+O(n^{-\eta}).\end{equation}
Assume the upper diagonal entries $(a_{ij})_{1\leq i\leq j\leq n}$ of $A_n$ are independent and identically distributed, and they satisfy $\mathbb{E}[a_{ij}]=0$, $\mathbb{E}[|a_{ij}|^2]=1$, and
\begin{equation}
   \lim_{x\to +\infty}x^{2(1+\mu^{-1})}\mathbb{P}(|a_{ij}|>x)=c\in(0,\infty).
\end{equation}
Let $\lambda_1(\mathcal{M}_n)\geq \lambda_2(\mathcal{M}_n)\geq\cdots\lambda_n(\mathcal{M}_n)$ denote the eigenvalues of $\mathcal{M}_n$, arranged in decreasing order. Then we have the following conclusion:

\begin{enumerate}
    \item (Top eigenvalue) We have the convergence in law as $n\to\infty$:
   \begin{equation}\label{bandtopconvergence} \lambda_1(\mathcal{M}_n) \overset{\text{law}}{\underset{n\to\infty}\longrightarrow} f(\xi_c^\mu),\end{equation} where the function $f$ is defined in \eqref{equationoff}, and the distribution of $\xi_c^\mu$ satisfies
\begin{equation}
    \mathbb{P}(\xi_c^\mu\leq x)=\exp\left(-\frac{cx^{-2(1+\frac{1}{\mu})}}{2}\right),\quad x>0.
\end{equation}

    \item (Joint distribution) More generally, for any $k>0$, the top $k$ eigenvalues of $\mathcal{M}_n$ have the joint distribution 
    \begin{equation}\label{bandjointlaws}
        (\lambda_1(\mathcal{M}_n),\cdots,\lambda_k(\mathcal{M}_n)) \overset{\text{law}}{\underset{n\to\infty}\longrightarrow} (f(\zeta_1),\cdots,f(\zeta_k)),
    \end{equation}
    where $\zeta_1\geq \cdots\geq \zeta_k>0$ are the $k$-largest points (in order) sampled from a Poisson point process on $[0,+\infty)$ with intensity measure $   \frac{c(1+\frac{1}{\mu})}{x^{2(1+\frac{1}{\mu})+1}}dx.$ The function $f$ is again defined by \eqref{equationoff}.
    \item (Eigenvector localization)  For each $k$ denote by $v_k$ the unit eigenvector of $\mathcal{M}_n$ corresponding to eigenvalue $\lambda_k=\lambda_k(\mathcal{M}_n)$. Consider, for any $\epsilon>0$, the event
     \begin{equation}
       \Omega_k^\epsilon:=\left( \begin{aligned}&\text{There exists distinct indices $i,j\in[1,n]^2$  }\\&\text{and some $D=\pm 1$ such that } 
       \\& \left|\langle v_k,\frac{1}{\sqrt{2}}(\delta_i+D\delta_j)\rangle\right|\in(1-\frac{4}{(\lambda_k+\sqrt{\lambda_k^2-4})^2})(1+[-\epsilon,\epsilon])
        \end{aligned} \right), 
    \end{equation} where $\frac{1}{\sqrt{2}}(\delta_i+D\delta_j)$ is the unit-norm vector in $\mathbb{R}^n$ which is non-zero only at its $i$ and $j$ coordinates, taking values $\frac{1}{\sqrt{2}}$ and  $D\frac{1}{\sqrt{2}}$ respectively.
    Then for any $\epsilon>0$ and $k>0$, 
    \begin{equation}\label{bandfinaljustification}
        \mathbb{P}\left(\Omega_k^\epsilon\mid \lambda_k>2\right){\underset{n\to\infty}\longrightarrow} 1.
    \end{equation}
\end{enumerate}
\end{theorem}

\begin{remark}
    In principle, we expect the same result holds for graphs with prescribed degree sequence $d_1,\cdots,p_n$ (growing at rate $n^{O(1)}$) such that $\frac{\max_id_i}{\min_id_i}\to 1$ as $n\to\infty.$ This would only amount to (straightforwardly) adapting the proof in \cite{au2023bbp} to this more general case, which is a BBP type result used in this paper. If however $\frac{\max_id_i}{\min_id_i}\to \alpha>1$, the limiting distribution will be different, as illustrated by the following example of sample covariance matrices. 
\end{remark}

\subsection{Sample covariance matrices}

Let $L,M$ be two integers, and $N=L+M$. Let $S_{L,M}$ be an $L\times M$ matrix with i.i.d. entries having mean $0$ and variance one. Consider the matrix $\frac{1}{L}(S_{L,M})(S_{L,M})^*$. Assuming that $\frac{M}{L}\to\alpha$, then under a second moment condition on entries of $S_{L,M}$, the spectral distribution of  $\frac{1}{L}(S_{L,M})(S_{L,M})^*$ converges to the Marchenko-Pastur law with density 
\begin{equation}\label{pasteurlaw} \pi_\alpha(dx)=\frac{\sqrt{(b_\alpha-x)(x-a_\alpha)}}{2\pi x}1_{[a_\alpha,b_\alpha]}(dx),
\end{equation}
when $\alpha>1$, and where $a_\alpha=(1-\sqrt{\alpha})^2,$ $b_\alpha=(1+\sqrt{\alpha})^2$.  
When $\alpha<1$ there is another spectral atom at $0$. Thus in this paper we focus on the case $\alpha>1$, since for the $\alpha<1$ case one only needs to consider $\frac{1}{M}(S_{L,M})^*(S_{L,M})$. 

When the individual entries of $S_{L,M}$ have finite fourth moment, then Bai and Silverstein \cite{bai1998no} proved that the largest eigenvalues of $\frac{1}{L}(S_{L,M})(S_{L,M})^*$ stick to the spectral edge. When the entries are regularly varying with index $\alpha<4$, then the largest eigenvalue has a Poisson distribution and almost surely deviates from the spectral edge\cite{article}. Also, Ding and Yang \cite{ding2018necessary} proved that a necessary and sufficient condition for the top eigenvalue of $\frac{1}{L}(S_{L,M})(S_{L,M})^*$ to have Tracy-Widom fluctuation is $\lim_{x\to\infty}x^4\mathbb{P}(|(S_{L,M})_{ij}|\geq x)=0$. To the author's best knowledge, the crossover case, i.e. when individual entries are regularly varying with index $4$, has not been investigated before for sample covariance matrices. We complete this picture by the following theorem:

\begin{theorem}\label{theorem1.9}
Let $S_{L,M}$ be a $L\times M$ matrix with i.i.d. elements that satisfy: $\mathbb{E}[(S_{L,M})_{ij}]=0,$ $\mathbb{E}[|(S_{L,M})_{ij}|^2]=1,$ and 
\begin{equation}
  \lim_{x\to\infty} x^4 \mathbb{P}(|(S_{L,M})_{ij}|\geq x)=c\in(0,\infty). 
\end{equation}
Assume that for some $\kappa>0$, $\frac{M}{L}=\alpha+O(N^{-\kappa})$. Set $N=L+M$. Then 
\begin{equation}
    \lambda_1\left(\frac{1}{L}(S_{L,M})(S_{L,M})^*\right) \overset{\text{law}}{\underset{L\to\infty}\longrightarrow} (1+\alpha)F_\alpha^2(\xi_{c,\alpha}),
\end{equation}
where $F_\alpha:\mathbb{R}_+\to[\frac{1+\sqrt{\alpha}}{\sqrt{1+\alpha}},\infty)$ is a deterministic function depending only on $\alpha$, defined in \eqref{functionoffalpha}; and $\xi_{c,\alpha}$ is a random variable that satisfies, for any $x>0$,
\begin{equation}
\mathbb{P}(\xi_{c,\alpha}\geq x)=e^{-\frac{c\alpha x^{-4}}{(1+\alpha)^2}}.
\end{equation}

\end{theorem}

It would also be possible to obtain the joint distribution of the finitely many largest eigenvalues of $ \frac{1}{L}S_{L,M}(S_{L,M})^*$, and the eigenvector overlaps. To achieve this goal, we need a more refined version of the finite rank perturbation result (Proposition \ref{proposition4.3} and \ref{proposition4.5}) covering eigenvector overlaps as well. The computations may get fairly complicated so we do not work them out in this paper.

From the complicated expressions involved in the distribution of $\lambda_1$ (to compute $F_\alpha(x)$, one needs to solve a cubic equation involving the Stieltjes transform of the Marchenko-Pastur law), one can expect that the sample covariance case in Theorem \ref{theorem1.9} is hardly accessible from purely method of moments computations, and this example fully illustrates the strength and efficiency of techniques in this paper.

\subsection{Plan of the paper} In Section \ref{1section2} we prove Theorem \ref{theorem1.2} and Theorem \ref{theorem1.2add} concerning sparse Wigner matrices. In Section \ref{1section3} we prove Theorem \ref{theorem1.7} concerning banded matrices and weighted regular graphs. In Section \ref{1section4} we prove Theorem \ref{theorem1.9} concerning sample covariance matrices.

\section{Proof for sparse Wigner matrices}\label{1section2}

We begin with some preliminary analysis about the matrix $\Xi_n$.

\subsection{Preliminary computations}
\begin{lemma}\label{lemma2.345} In the setting of Theorem \ref{theorem1.2}, the following estimates are satisfied:
    \begin{enumerate}
        \item (Diagonal elements) We have 
        \begin{equation}\label{equationcase1}
            \mathbb{P}\left(\text{ there exists $i\in[1,n]$ such that } |(\Xi_n)_{ii}|>\frac{1}{(\log n)^5}\right)=o(1).
        \end{equation}
        \item (Off-diagonal elements) We have 
         \begin{equation}\label{equationcase2}
            \mathbb{P}\left(
        \begin{aligned}&\text{ there exists some $i\in[1,n]$ such that for two $j\neq k$, }\\& |(\Xi_n)_{ij}|>\frac{1}{(\log n)^5},\quad |(\Xi_n)_{ik}|>\frac{1}{(\log n)^5}\end{aligned}\right)=o(1).
        \end{equation}
        \item (Largest value) For any $\delta>0$, 
          \begin{equation}\label{equationcase3}
         \lim_{n\to\infty} \lim_{M\to\infty}
            \mathbb{P}\left(\text{ there are $M$ subscripts $(i_s,j_s)$ such that } |(\Xi_n)_{i_sj_s}|>\delta\right)=0.
        \end{equation}
    \end{enumerate}
\end{lemma}

\begin{proof} For any $(i,j)\in[1,n]^2,$ we bound for any $x>0$,
$$\mathbb{P}(|(\Xi_n)_{ij}|>x)=\mathbb{P}(b_{ij}=1)\mathbb{P}(|a_{ij}|>x\sqrt{p_n})\to\frac{p_n}{n}\times cx^{-2(1+\frac{1}{\mu})}(p_n)^{-(1+\frac{1}{\mu})},$$
where we used $p_n\to \infty$ as $n\to\infty$. Thus by a union bound, the left hand side of \eqref{equationcase1} is upper bounded by 
$$
n\times \frac{p_n}{n}\times c(\log n)^{10(1+\frac{1}{\mu})}(p_n)^{-(1+\frac{1}{\mu})}{\underset{n\to\infty}\longrightarrow}0.$$
By a union bound, the left hand side of \eqref{equationcase2} is upper bounded by 
$$
n^3\times (\frac{p_n}{n})^2\times c^2(\log n)^{20(1+\frac{1}{\mu})}(p_n)^{-2(1+\frac{1}{\mu})}{\underset{n\to\infty}\longrightarrow}0 .$$
And finally, by a union bound, the left hand side of \eqref{equationcase3} is upper bounded by
$$
\frac{n^M}{M!}\times n^{-M}(c\delta^{-2(1+\frac{1}{\mu})})^M{\underset{M\to\infty}\longrightarrow}0.
$$
\end{proof}

Likewise, we derive the following estimate:

\begin{lemma}\label{lemma2.34567}In the setting of Theorem \ref{theorem1.2}, let $T_1\geq T_2\geq \cdots$ denote the set $\{|(\Xi_n)_{ij}|,1\leq i\leq j\leq n\}$ with elements re-arranged in decreasing order. Then in the case $p_n=n^\mu$, we have for any $x>0$,
$$
\mathbb{P}(T_1\leq x)=\exp\left(-\frac{cx^{-2(1+\frac{1}{\mu})}}{2}\right),
$$ and for any $c>0$, the set $\{|(\Xi_n)_{ij}|1\leq i\leq j\leq n\}\cap [c,\infty)$  converges to the Poisson point process on $[c,\infty)$ with intensity measure
\begin{equation}\label{intensitymeasures}
    \frac{c(1+\frac{1}{\mu})}{x^{2(1+\frac{1}{\mu})+1}}dx.
\end{equation}

\end{lemma}

\begin{proof}
Recall that for each $(i,j)$, $\mathbb{P}(|(\Xi_n)_{ij}|>x)=\frac{cx^{-2(1+\frac{1}{\mu})}}{n^2},$ so that in the regime $p_n=n^\mu$, taking a union bound for all $1\leq i\leq j\leq n$,
$$
\mathbb{P}(T_1\leq x)=\prod_{1\leq i\leq j\leq n} \mathbb{P}(|(\Xi_n)_{ij}|<x)=(1-\frac{cx^{-2(1+\frac{1}{\mu})}}{n^2})^\frac{n^2+n}{2}\to \exp(-\frac{cx^{-2(1+\frac{1}{\mu})}}{2}).
$$
To prove the second claim, one only need a straightforward generalization of the above computations, see \cite{leadbetter2012extremes}, Theorem 2.3.1. This claim is also exactly the same as in \cite{Soshnikov2004},Proposition 1.
\end{proof}

\subsection{Sampling procedure}\label{sampling1st} In this section we discuss a method to sample the random matrix $\Xi_n$ differently, so that $\Xi_n$ can be expressed as the sum of a matrix with small entries and having the same variance profile as a GOE, and another sparse matrix with large entries. This idea of resampling is inspired by \cite{aggarwal2019bulk}.

\begin{Sampling} In the setting of both Theorem \ref{theorem1.2} and Theorem \ref{theorem1.2add}, we take the following steps to resample the matrix $\Xi_n$: we abbreviate $Q_n:=\sqrt{p_n}(\log n)^{-5}$.

\begin{enumerate}
\item 
 Initially, sample the sparsity matrix $B_n$ such that $b_{ij}=b_{ji}$ and that $b_{ij}$ is distributed as $\operatorname{Ber}(\frac{p_n}{n})$ for each $1\leq i\leq j\leq n$. 
\item Independently for each $(i,j)\in[1,n]^2$, $i\leq j$, such that $g_{ij}=1$, we consider a Bernoulli random variable which is 0  with probability $\mathbb{P}(|a_{ij}|<Q_n)$, and is 1 otherwise. If the Bernoulli variable takes the value 0, we place a label $S$ at site $(i,j)$; while if the Bernoulli variable takes value 1, we place a label $L$ at site $(i,j)$. Here $S$ stands for small and $L$ stands for large.
    
\item Then for each $1\leq i\leq j\leq n$ such that $(i,j)$ is assigned with a label $L$, we independently sample $\widehat{a}_{ij}$ from the law $\mathbb{P}(a_{ij}\in dx\mid |a_{ij}|>Q_n)$ (that is, the law of $a_{ij}$ conditioned on the event that $|a_{ij}|>Q_n$) and then set $\widehat{a}_{ij}=\widehat{a}_{ji}$; and for $(i,j)$ assigned with label $S$ or for $(i,j)$ such that $b_{ij}=0$, we simply set $\widehat{a}_{ij}=0.$
\item From this, define for $(i,j)\in[1,n]^2:$ $$(\Xi_n^{L})_{ij}:=\frac{1}{\sqrt{p_n}}b_{ij}\widehat{a}_{ij}.$$ 
\item For each $1\leq i\leq j\leq n$ let $\overline{a}_{ij}$ be i.i.d. random variables (independent of the $\widehat{a}_{ij}$'s) having distribution $\mathbb{P}(a_{ij}\in dx\mid |a_{ij}|<Q_n)$ (that is, the law of $a_{ij}$ conditioned on the event that $|a_{ij}|<Q_n$), then set $\overline{a}_{ij}=\overline{a}_{ji}$ and define
$$(\Xi_n^{S})_{ij}:=\frac{1}{\sqrt{p_n}}b_{ij}\overline{a}_{ij}.$$

\item Finally,  set $$(\Xi_n^{N})_{ij}:=-\frac{1}{\sqrt{p_n}}b_{ij}\overline{a}_{ij}1_{\text{label of $(i,j)$ is $L$}}.$$ We observe from our construction that
\begin{equation}\label{equation2.52.5}
    \Xi_n\overset{\text{law}}{\underset{}=}\Xi_n^S+\Xi_n^L+\Xi_n^N,
\end{equation}
and that conditioned on the filtration $\mathcal{F}_{B,S,L}$ generated by sampling the matrix $B$ and by sampling the labels $S$ and $L$, the two random matrices $\Xi_n^L$ and $\Xi_n^S$ have independent entries, which is a crucial property we will use in the sequel.

To check \eqref{equation2.52.5}, we only need to notice that, given that $\widehat{a}_{ij}$ and $\overline{a}_{ij}$ are sampled independently, $$
a_{ij}\overset{\text{law}}{\underset{}=} \widehat{a}_{ij}1_{(i,j)\text{ has label L}}+\overline{a}_{ij}1_{(i,j)\text{ has label S}}.
$$

\end{enumerate}    
\end{Sampling}

We next show that $\Xi_n^N$ is negligible in the limit when we consider the top eigenvalues:

\begin{lemma} Under the assumption of Theorem \ref{theorem1.2}, for each sampled matrix $B$ and sampled label configurations $S,L$ we use $\Omega_{B,S,L}$ to denote the event that the sampling procedure yields $B,S,L$ as the output. Then there exists a Borel subset $\Omega_n\subset \Omega$, $\Omega_n\in\mathcal{F}_{B,S,L}$ with $\mathbb{P}(\Omega_n)\to 1$ such that for any $\Omega_{B,S,L}\in\Omega_n$,  we have the following almost sure convergence on $\Omega_n$
\begin{equation}
   | \lambda_1(\Xi_n^N)|\leq (\log n)^{-5}{\underset{n\to\infty}\longrightarrow}0,\quad\text{ on } \Omega_n,
\end{equation}
and that on this event $\Omega_n$, for any $k\in\mathbb{N}_+,$
\begin{equation}\label{secondclaim}
 |\lambda_k(\Xi_n^N+\Xi_n^S+\Xi_n^L)-  \lambda_k(\Xi_n^S+\Xi_n^L)|\leq (\log n)^{-5}{\underset{n\to\infty}\longrightarrow} 0,
\end{equation} where $\lambda_k(\cdot)$ denotes the $k$-th largest eigenvalue of the given matrix.
\end{lemma}

\begin{proof}
By Lemma \ref{lemma2.345}, with probability tending to one, on each row and column of $\Xi_n^N$ there is at most one non-zero element. So that upon rearranging the rows and columns, we may assume $\Xi_n^N$ is in a diagonal block form and all the eigenvalues of $\Xi_n^N$ correspond to the non-zero elements of $\Xi_n^N$, which are deterministically less than $(\log n)^{-5}$ (in absolute value) by definition of $\overline{a}_{ij}$. Therefore, taking $\Omega_n$ the event that on each row and column of $\Xi_n^N$ there is at most one non-zero element, completes the proof of the lemma.

The second claim \eqref{secondclaim} follows from Weyl interlacing formula for symmetric matrices.
\end{proof}

We finally show that only the largest elements in $\Xi_n^L$ need to be taken into account when analyzing edge eigenvalues.

\begin{lemma} 
 For any $c>0$, define the following two matrices \begin{equation}
    (\Xi^{L,\geq c}_n)_{ij}= (\Xi^{L}_n)_{ij}1_{|(\Xi^{L}_n)_{ij}|\geq c},\quad 
    (\Xi^{L,\leq c}_n)_{ij}= (\Xi^{L}_n)_{ij}1_{|(\Xi^{L}_n)_{ij}|\leq c}.
\end{equation}
Then there exists $\Omega_n\in\mathcal{F}_{B,S,L},$ $\mathbb{P}(\Omega_n)\to 1$, such that for any $(B,S,L)\in \Omega_n$ we have
\begin{equation}
    |\lambda_1(\Xi^{L,\leq c}_n)|\leq c,
\end{equation}
and that for any $(B,S,L)\in \Omega_n$, for any $k\in\mathbb{N}_+,$
\begin{equation}\label{modifiedsecondclaim}
 |\lambda_k(\Xi_n^N+\Xi_n^S+\Xi_n^L)-  \lambda_k(\Xi_n^S+\Xi_n^{L,\geq c})|\leq (\log n)^{-5}+c.
\end{equation}
\end{lemma}

The proof of this lemma is exactly the same as before: simply note that for $\Xi_n^L$, on the event $\Omega_n$ defined as in the previous lemma, there is at most one nonzero element in each row and column of $\Xi_n^L$, so that $\|\Xi_n^{L,\leq c}\|\leq c$.

To conclude, in order to determine the edge eigenvalue of $\Xi_n$, it suffices to determine the edge eigenvalue of $\Xi_n^S+\Xi_n^{L,\geq c}$, and show it is consistent as we send $c\to 0$.

\subsection{Universality principle and proof for eigenvalue statistics}
From our sampling procedure, after sampling $B$,$S$ and $L$, the two components $\Xi_n^{L,\geq c}$ and $\Xi_n^{S}$ have independent entries conditioned on $\mathcal{F}_{B,S,L}$. Thus for our goal, it suffices to sample $\Xi_n^{L,\geq c}$ first, and then the edge eigenvalues of $\Xi_n^{S}+\Xi_n^{L,\geq c}$ should follow from finite rank perturbation results applied to $\Xi_n^{S}$. However, although the covariance structure of $\Xi_n^{S}$ is close to a Gaussian random matrix, its individual elements do not have sufficiently many moments for classical finite rank perturbation results to be applied.

The recent breakthrough work of Brailovskaya and Van Handel \cite{brailovskaya2022universality} provides a user-friendly solution to this problem. In the setting of \cite{brailovskaya2022universality}, they consider a random matrix
\begin{equation}\label{formofX}
    X:=Z_0+\sum_{i=1}^n Z_i, 
\end{equation} where for some $d\in\mathbb{N}_+,$ $Z_0\in \operatorname{M}_d(\mathbb{C})_{sa}$ is a deterministic matrix and $Z_1,\cdots,Z_n$ are i.i.d. self-adjoint matrices in $\operatorname{M}_d(\mathbb{C})_{sa}$ that satisfy $\mathbb{E}[Z_i]=0.$ For this matrix they consider the following four parameters (we always use the symbol $\|\cdot\|$ to denote the operator norm of a matrix)
\begin{equation}
    \sigma(x):=\|\mathbb{E}[X-\mathbb{E}[X]]^2\|^\frac{1}{2},
\end{equation}

\begin{equation}
    \sigma_*(X):=\sup_{\|v\|=\|w\|=1} \mathbb{E}\left[|\langle v,(X-\mathbb{E}[X])w\rangle|^2\right]^\frac{1}{2},
\end{equation}

\begin{equation}
    v(x):=\|\operatorname{Cov}(X)\|^\frac{1}{2},
\end{equation}

\begin{equation}
    R(x):=\|\max_{1\leq i\leq n}\|Z_i\|\|^\frac{1}{2}.
\end{equation}

Then we have the following result from \cite{brailovskaya2022universality}:
\begin{theorem}(\cite{brailovskaya2022universality}, Theorem 3.31)\label{theorem2.7} Let $G_d$ be a $d\times d$ self-adjoint random matrix whose entries $({G_{d}}_{ij})_{i\geq j}$ have independent real Gaussian distributions with mean $0$ and variance $\frac{1+1_{i=j}}{d}$. Let $A_d$ be deterministic $d\times d$ matrix with eigenvalues $\theta_1\geq\cdots\geq \theta_r\geq0\geq \theta_{r+1}\geq \cdots\geq\theta_s$ that are independent of $d$. Assume that $H_d$ is a random matrix of the form \eqref{formofX} whose entries have the same mean and variance of $G_d$, and that $(\log d)^2R(H_d)\to 0$ as $d\to\infty$. Then the following conclusions hold: (recall the function $f$ defined in \eqref{equationoff})
\begin{enumerate}
    \item For $1\leq i\leq r$, $\lambda_i(A_d+H_d){\underset{n\to\infty}\longrightarrow} f(\theta_i),$ and $\lambda_i(A_d+H_d){\underset{n\to\infty}\longrightarrow} 2$ for $i\geq r+1,$
    \item For each $1\leq i,j\leq r$ such that $\theta_j\neq \theta_i$ and $\theta_i>1$, 
\begin{equation}\label{theorem2.7case2}
    \|P_i(A_d)v_i(A_d+H_d)\|^2{\underset{n\to\infty}\longrightarrow} 1-\frac{1}{\theta_i^2},\quad  \|P_j(A_d)v_i(A_d+H_d)\|^2{\underset{n\to\infty}\longrightarrow} 0,
\end{equation}
    where for a matrix $M$, $P_i(M)$ denotes the projection onto the eigenspace of $M$ corresponding to the eigenvalue $\lambda_i(M)$, and $v_i(M)$ is the unit eigenvector of $M$ with eigenvalue $\lambda_i(M)$.
\end{enumerate}
    
\end{theorem}
Note that in the statement of \cite{brailovskaya2022universality}, Theorem 3.31 only matrix $A_n$ with finitely many non-negative eigenvalues were considered, but the claim clearly extends to the current setting with $A_n$ having both finitely many positive and negative eigenvalues.

Now we complete the proof of Theorem \ref{theorem1.2}.

\begin{proof}[\proofname\ of Theorem \ref{theorem1.2}, claim (1) and (2)] We first sample the sparse matrix $B$, the labels $S$ and $L$, and it suffices to further sample the matrix sum $\Xi_n^S+\Xi_n^{L,\geq c}$. Thanks to independence, we first sample $\Xi_n^{L,\geq c}$. By Lemma \ref{lemma2.34567}, the non-zero elements in $\{|(\Xi_n^{L,\geq c})|_{i,j},1\leq i\leq j\leq n\}$ converge to in distribution to the Poisson point process on $[c,\infty)$ with intensity measure \eqref{intensitymeasures}, and almost surely there are only finitely many such points. Also by Lemma \ref{lemma2.345}, with probability tending to one, all diagonal elements of ${\Xi_n^{L,\geq c}}$ vanish, and no row or column of ${\Xi_n^{L,\geq c}}$ has at least two non-zero elements. Denote by $\zeta_1\geq\cdots\geq \zeta_r\geq c$ the points in this Poisson point process. Also denote by $\mathcal{F}_{B,S,L,\geq c}$ the filtration generated by sampling $B$, labels $S,L$, and sampling $\Xi_n^{L,\geq c}$.

Now, conditioned on the filtration $\mathcal{F}_{B,S,L,\geq c}$, the real symmetric matrix $\Xi_n^S$ has i.i.d entries on and above diagonal, with variance $\frac{\operatorname{Var}\left(a_{ij}1_{|a_{ij}|\leq Q_n}\right)}{n}$. Further, from our sampling procedure, we have almost surely $(\log n)^2 R(\Xi_n^S)\to 0$ as $n\to\infty$.

We now wish to apply Theorem \ref{theorem2.7}, (1) with the choice $H_d=\Xi_n^S$ and $A_d=\Xi_n^{L,\geq c}$. We have to make the following modifications: since $\lim_{n\to\infty}\operatorname{Var}\left(a_{ij}1_{|a_{ij}|\leq Q_n}\right)\to 1$, we shall consider $a_n\Xi_n^S$ where $a_n>1,a_n\to 1$ so that each off-diagonal element of $a_n\Xi_n^S$ has variance $\frac{1}{n}$. Also, even though the diagonal element of $a_n\Xi_n^S$ does not have variance $\frac{2}{n}$, we can add an independent diagonal matrix $L^1_n$ so that variance profile of $a_n\Xi_n^S+L^1_n$ satisfy the assumptions of Theorem \ref{theorem2.7}. Finally, to remedy the fact that the eigenvalues of $\Xi_n^{L,\geq c}$ are not fixed but only converge to $(\zeta_1,\cdots,\zeta_r)$ in the limit, we add another matrix $L^2_n$ in such a way that $\|L^2_n\|\to 0$, and moreover that eigenvalues of $\Xi_n^{L,\geq c}+L^2_n$ are independent of $n$ for $n$ sufficiently large. This matrix $L^2_n$ is necessarily random, but is measurable with respect to $\mathcal{F}_{B,S,L,\geq c}$ and independent of $\Xi_n^S$ .Now all conditions of Theorem \ref{theorem2.7} have been verified and claim (1) of Theorem \ref{theorem2.7} holds for $a_n\Xi_n^S+L_n^1+L_n^2+\Xi_n^{L,\geq c}$ in place of $A_d+H_d$. 
Then Theorem \ref{theorem2.7} implies the following convergence in distribution
$$
\lambda_k(a_n\Xi_n^S+L_n^1+L_n^2+\Xi_n^{L,\geq c}){\underset{n\to\infty}\longrightarrow} f(\zeta_k) .
$$
To turn back to $\Xi_n^S+\Xi_n^{L,\geq c}$, note that the operator norms of $L^1_n,L^2_n$ and $(a_n-1)\Xi_n^S$ all converge to $0$ in probability as $n\to\infty$, so that we have the convergence in probability
$$
\|(a_n-1)\Xi_n^S+L_n^1+L_n^2\|{\underset{n\to\infty}\longrightarrow} 0,
$$
and use Weyl interlacing formula for eigenvalue of symmetric matrices to deduce $$
\left|\lambda_k(a_n\Xi_n^S+L_n^1+L_n^2+\Xi_n^{L,\geq c})-\lambda_k(\Xi_n) \right|\leq \|\Xi_n^{L,\leq c}\|+\|(a_n-1)\Xi_n^S+L_n^1+L_n^2\|{\underset{n\to\infty,c\to 0}\longrightarrow} 0.
$$ We actually first set $n\to\infty$ for fixed $c$, and then set $c\to 0$. This concludes the proof of Theorem \ref{theorem1.2}, cases (1) and (2).
\end{proof}

\subsection{Eigenvector statistics}
Since $a_{ij}$ has a symmetric law, then we see that the non-zero elements of $(\Xi_n^{L,\geq c})_{ij}$ converge to the Poisson point process on $(-\infty,\infty)\setminus[-c,c]$ with density measure $\frac{c(1+\frac{1}{\mu})}{2x^{2(1+\frac{1}{\mu})+1}}dx.$ Recall in the proof in the previous paragraph that we added $L^n_2$ so that $\Xi_n^{L,\geq c}+L_n^2$ has eigenvalue independent of $n$, and that $\|L_n^2\|\to 0$, so one sees that as long as $n$ is sufficiently large, the eigenspaces corresponding to non-zero eigenvalues of $L_n^2+\Xi_n^{L,\geq c}$ and $\Xi_n^{L,\geq c}$ coincide.
More precisely, suppose some nonzero element $(\Xi_n^{L,\geq c})_{i_n,j_n}$ with index $(i_n,j_n)$ in $\Xi_n^{L,\geq c}$ converge to $\bar{\xi}$, then the $(i_n,j_n)$ elements of $\Xi_n^{L,\geq c}+L_n^2$ also converge to $\bar{\xi}$; and when $\bar{\xi}>0$, then $(0,\cdots,\frac{1}{\sqrt{2}},\cdots,\frac{1}{\sqrt{2}},\cdots,0)$ (being nonzero on the $i_n$ and $j_n$ entries, and zero elsewhere) is an eigenvector for both matrices, corresponding to eigenvalues converging to $\bar{\xi}$. When $\bar{\xi}<0$, the vector $(0,\cdots,\frac{1}{\sqrt{2}},\cdots,-\frac{1}{\sqrt{2}},\cdots,0)$ (non-zero only on $i_n$ and $j_n$ entries) is an eigenvector of both matrices, associated to eigenvalues converging to $-\bar{\xi}$. This implies, at least for $n$ sufficiently large, the two projections $P_i(L_n^2+\Xi_n^{L,\geq c})$ and $P_i(\Xi_n^{L,\geq c})$  are identical for each fixed $i$, where we recall that $P_i(\cdot)$ is the projection onto the linear space spanned by the eigenvector of the $i$-th largest eigenvalue of the given matrix.

To complete the proof of Theorem \ref{theorem1.2}, part (3), we need the following simple lemma.

\begin{lemma}\label{lemmaline422}
In the setting of Theorem \ref{theorem2.7}, assume that a sequence of $n\times n$ symmetric random matrices $H_n$ are given such that \eqref{theorem2.7case2} holds almost surely. Then given any $n\times n$ symmetric matrix $L_n$ with $\|L_n\|\overset{\text{a.s.}}{\underset{n\to\infty}\longrightarrow} 0$, one must have \eqref{theorem2.7case2} holds verbatim with  $H_n+L_n$ in place of $H_n$.
\end{lemma}

\begin{proof}   Recall that $v_n^j$ is the unit eigenvector of $A_n$ corresponding to $\theta_j$.
For any smooth, compactly supported function $\varphi:\mathbb{R}\to [0,1]$, we have almost surely 
\begin{equation}\label{agagagagaga}
\langle v_n^j, \varphi(H_n+A_n+L_n)v_n^j\rangle-\langle v_n^j, \varphi(H_n+A_n)v_n^j\rangle {\underset{n\to\infty}\longrightarrow} 0.
\end{equation}
 Now we take $\varphi^i_\epsilon$ a non-negative smooth function supported on $[\theta_i+\frac{1}{\theta_i}-\epsilon,\theta_i+\frac{1}{\theta_i}+\epsilon]$, identical to one on $[\theta_i+\frac{1}{\theta_i}-\frac{\epsilon}{2},\theta_i+\frac{1}{\theta_i}+\frac{\epsilon}{2}]$. By definition, we have
 $$ 1_{[\theta_j+\frac{1}{\theta_j}-\frac{\epsilon}{2},\theta_j+\frac{1}{\theta_j}+\frac{\epsilon}{2}]}(H_n+A_n)\leq 
\varphi_\epsilon^j(H_n+A_n) \leq 1_{[\theta_j+\frac{1}{\theta_j}-\epsilon,\theta_j+\frac{1}{\theta_j}+\epsilon]}(H_n+A_n)
 $$ and 
 $$ 1_{[\theta_j+\frac{1}{\theta_j}-\frac{\epsilon}{2},\theta_j+\frac{1}{\theta_j}+\frac{\epsilon}{2}]}(H_n+A_n+L_n)\leq 
\varphi_\epsilon^j(H_n+A_n+L_n) \leq 1_{[\theta_j+\frac{1}{\theta_j}-\epsilon,\theta_j+\frac{1}{\theta_j}+\epsilon]}(H_n+A_n+L_n).
 $$
   By part (2) of Theorem \ref{theorem2.7}, $H_n+A_n$ has an eigenvalue converging to $\frac{1}{\theta_j}+\theta_j$, and so does $H_n+A_n+L_n$. Now we assume that $\theta_1,\cdots,\theta_s$ are mutually distinct (the only case we will use in this paper). Then given $\epsilon>0$ is sufficiently small, $\langle v_n^j,1_{[\theta_j+\frac{1}{\theta_j}-\epsilon,\theta_j+\frac{1}{\theta_j}+\epsilon]}(H_n+A_n)v_n^j\rangle$ is equal to the inner product of $v_n^j$ with the eigenvector of $H_n+A_n$ associated to the eigenvalue converging to $\theta_j+\frac{1}{\theta_j}$ as $n$ tends to infinity, and so does $H_n+A_n+L_n$. Then the claim follows from \eqref{agagagagaga}, that is, we first take $\epsilon$ fixed but sufficiently small, then set $n$ to infinity and finish the proof.

   In the case where some of $\theta_1,\cdots,\theta_s$ are equal, one may use a perturbation argument as in the proof of \cite{brailovskaya2022universality}, Corollary 9.27. We do not present the details as not needed here.
\end{proof}
Now we can complete the proof of Theorem \ref{theorem1.2}.

\begin{proof}[\proofname\ of Theorem \ref{theorem1.2}, claim (3)] The claim is a direct consequence of Theorem \ref{theorem2.7}, combined with the sampling procedure, the point process convergence in Lemma \ref{lemma2.34567}, and finally using Lemma \ref{lemmaline422} to remove perturbations of vanishing operator norm. 

To check the resulting expression in \eqref{whatistheevents}, note that for the spiked eigenvalue to converge to $\lambda_k$, the spike needs to converge to $\frac{\lambda_k+\sqrt{\lambda_k^2-4}}{2}$. By previous discussions the spike is associated (with probability 1-o(1)) to an eigenvector of the form $\frac{1}{\sqrt{2}}(\delta_i+D\delta_j)$ for some random $i\neq j$ and $D=\pm 1$. Then \eqref{whatistheevents} follows from applying Theorem \ref{theorem2.7}, part (2) and Lemma \ref{lemmaline422}.
    
\end{proof}

\subsection{The regime of super-polynomial decay}
We begin with a computational lemma:

\begin{lemma}\label{exponentiallemma2.345} In the setting of Theorem \ref{theorem1.2add}, the following estimates are satisfied:
    \begin{enumerate}
        \item (Diagonal elements) We have 
        \begin{equation}\label{}
            \mathbb{P}\left(\text{ there exists $i\in[1,n]$ such that } |(\Xi_n)_{i,i}|>\frac{1}{(\log n)^5}\right)=o(1).
        \end{equation}
        \item (Off-diagonal elements) We have 
         \begin{equation}\label{}
            \mathbb{P}\left(
        \begin{aligned}&\text{ there exists some $i\in[1,n]$ such that for two $j\neq k$, }\\& |(\Xi_n)_{i,j}|>\frac{1}{(\log n)^5},\quad |(\Xi_n)_{i,k}|>\frac{1}{(\log n)^5}\end{aligned}\right)=o(1).
        \end{equation}
          \item (Largest element)  Denote by $T_1:=\max(\{|(\Xi_n)_{ij}|,i,j=1,\cdots,n\})$, then we have a deterministic limit
          \begin{equation}
              T_1\overset{\text{law}}{\underset{n\to\infty}\longrightarrow} a.
          \end{equation}
    \end{enumerate}
    \end{lemma}

\begin{proof}

In this case $p_n=(\log n)^{g(n)}$   where $g(n)$ satisfies \eqref{whatisg}. We first derive a simple estimate that relies on the assumption \eqref{whatisg}: for any $\gamma>0$, we have
\begin{equation}\label{fucequaiton}
\frac{p_n}{(\log n)^\gamma}\geq p_{n^{0.99}}.
\end{equation}
To prove this, simply note that 
$$p_{n^{0.99}}=(0.99\log n)^{g(n^{0.99})}\leq 0.99^{g({n^{0.99}})} p_n,$$
and that $$g(n^{0.99})=\Omega(\log\log(n^{0.99}))\geq \log(0.99\log n)\geq \log\log n+\log 0.99,$$ so that $g({n^{0.99}})=\Omega((\log\log n))$, and the proof of \eqref{fucequaiton} is complete.

In this case we have, for any $(j,k)$:
$$\begin{aligned}
\mathbb{P}(|(\Xi_n)_{jk}|\geq (\log n)^{-5})&= \frac{p_n}{n} \mathbb{P}(|a_{jk}|\geq \sqrt{p_n}(\log n)^{-5})\\&\leq \frac{p_n}{n}\mathbb{P}(|a_{jk}|\geq \sqrt{p_{n^{0.99}}})\\&\leq \frac{p_n}{n}\times (p_n)^{-0.99}\times h(\sqrt{p_n^{0.99}}) \\&\leq \frac{(p_n)^{0.01}}{n^{1.99}}.\end{aligned}$$
Thus the first two claims are proved via a trivial union bound.

For the claim in (3), we need a different estimate. Consider the following problem.
Knowing that
$$
(\log x)^{g(x)}=m,
$$
what is the solution to 
$$
(\log y)^{g(y)}=3m?
$$
The solution to the first equation satisfies 
$g(x)\log\log x=\log n$, so that $y$ should satisfy
$g(y)\log\log y=\log 3n$. But the assumption \eqref{regularityg} forces $\log(y)-\log(x)=\Omega(\log 3),$ so that $\frac{y}{x}\to +\infty$ as $x\to\infty$. Recall that $h(x)$ satisfies $h(a\sqrt{\log(x)^{g(x)}})=x,$ that is, $h(a\sqrt{m})=x,$ and $h(a\sqrt{3}\sqrt{m})=y.$ Noting that $y/x\to\infty$, that the estimate holds for all $m>0$ sufficiently large, and that the constant $3$ can be replaced by any number larger than one, we deduce that for any $C>1$, we have
\begin{equation}
    \lim_{m\to\infty}\frac{h(Cm)}{h(m)}=+\infty.
\end{equation}

Now we can finish the proof. 
\begin{equation}\label{limitsup}
    \mathbb{P}(T_1<x)=((\frac{p_n}{n}\mathbb{P}(a_{ij})<\sqrt{p_n}x))^\frac{n(n+1)}{2}=(1-\frac{1}{h(\sqrt{p_n}x)n})^\frac{n^2+n}{2}.
\end{equation}
    Note that $h(\sqrt{p_n}a)=n$, and by the previous computation, for any $b>a$, $h(\sqrt{p_n}b)=\Omega(n),$ and for any $c<a$, we have  $h(\sqrt{p_n}c)=o(n),$ taking the limit in \eqref{limitsup} completes the proof in part (3).
\end{proof}

Applying the sampling procedure in Section \ref{sampling1st} to $\Xi_n$, we obtain the decomposition
$$
\Xi_n=\Xi_n^S+\Xi_n^{L,\leq c}+\Xi_n^{L,\geq c}
$$
where $\Xi_n^{L,\geq c}$ contains elements in $\Xi_n$ with absolute value larger than $c$. By the previous lemma, with probability tending to one, in each row or column there are at most one nonzero elements in $\Xi_n^{L,\geq c}$ and all diagonal elements are $0$, so that the nonzero eigenvalues of $\Xi_n^{L,\geq c}$ are given by nonzero elements in $\Xi_n^{L,\geq c}$ times $\pm 1$. Moreover, $\Xi_n^{L,\geq c}$ has rank $o(n)$, but a notable difference compared to the denser case $p_n=n^{O(1)}$ is that $\operatorname{rank}(\Xi_n^{L,\geq c})>>1$ when $c$ is sufficiently small. This forces us to use a slightly different mesoscopic perturbation result.

\begin{lemma}(See for example \cite{huang2018mesoscopic}) Let $W_n$ be an $n\times n$ GOE matrix, and $A_n$ be a deterministic $n\times n$ matrix with rank $o(n)$ having its largest eigenvalue $\lambda_1>0$. Then the largest eigenvalue of $W_n+A_n$ converges in distribution to $f(\lambda_1)$ as $n\to\infty$.
\end{lemma}
Although we have not verified a separation condition stated in \cite{huang2018mesoscopic}, this is irrelevant to us as we only consider the largest eigenvalue. 

The stated mesoscopic perturbation result can be immediately generalized to $\Xi_n^{S}$:

\begin{Proposition}\label{mesoscopicfora} In the setting of Theorem \ref{theorem1.2add}, let $A_n$ be deterministic $n\times n$ matrix with rank $o(n)$ having its largest eigenvalue $\lambda_1>0$. Then the largest eigenvalue of $\Xi_n^{S}+A_n$ converges in distribution to $f(\lambda_1)$ as $n\to\infty$.
\end{Proposition}

\begin{proof} Again find a sequence $a_n\to 1$ such that each element of $a_n\Xi_n^S$ has variance $\frac{1}{n}$, thus having the same mean and variance profile as a GOE matrix $W_n$ (except the diagonal entries, which can be addressed after adding a small perturbation with vanishing norm). Now we apply the universality principle in \cite{brailovskaya2022universality}, Theorem 2.7 to $a_n\Xi_n^s+A_n$ and $W_n+A_n$. The most important condition to check in that theorem is that $R(a_n\Xi_n^s+A_n)=O(\log n)^{-5}{\underset{n\to\infty}\longrightarrow} 0$, which follows from our construction of $\Xi_n^s$, and that $\sigma_*(a_n\Xi_n^s+A_n)\leq n^{-\Omega(1)}$ from a standard computation.  Thus by \cite{brailovskaya2022universality}, Theorem 2.7, we have that $d_H(\operatorname{Sp}(a_n\Xi_n^S+A_n),\operatorname{Sp}(W_n+A_n))\to 0$ in probability as $n\to\infty$, where $d_H$ is the Hausdorff distance between subsets of $\mathbb{R}$, and $\operatorname{Sp}(\cdot)$ denotes the spectrum of the matrix. This implies convergence of top eigenvalue: $|\lambda_1(W_n+A_n)-\lambda_1(a_n\Xi_n^S+A_n)|\to 0$ as $n\to\infty$. Finally, use $\|(a_n-1)\Xi_n^S\|\to 0$ and Weyl interlacing inequality to conclude the proof. (That $\|\Xi_n^S\|$ is bounded with high probability also follows from \cite{brailovskaya2022universality}, Corollary 2.6.)
\end{proof}

We have now collected all pieces for the proof of Theorem \ref{theorem1.2add}.

\begin{proof}[\proofname\ of Theorem \ref{theorem1.2add}] The claim follows from combining the following three ingredients: (a) the mesoscopic perturbation result in Proposition \eqref{mesoscopicfora}; (b) the sampling procedure in Section \ref{sampling1st}; and (c) the fact that the largest absolute value among all elements of $\Xi_n$ converge in probability to the deterministic constant $a$, which is implied by Lemma \ref{exponentiallemma2.345}, (3).
    
\end{proof}

\section{Proof for weighted regular graphs}\label{1section3}
In this section we prove Theorem \ref{theorem1.7}. 

\subsection{Sampling and initial estimate}

We begin with a sampling procedure that is very similar to Section\ref{sampling1st}. In this case we define a constant $Q_n$ as $Q_n:=\sqrt{k_n}(\log n)^{-5}.$

\begin{Sampling}\label{sampling3.1} In the setting of Theorem \ref{theorem1.7}, we take the following steps to resample $\mathcal{M}_n$:
\begin{enumerate}
\item First, independently for each $(i,j)\in[1,n]^2$, $i\leq j$, such that $g_{ij}=1$, we place a label $S$ with probability $\mathbb{P}(|a_{ij}|<Q_n)$, and place a label $L$ with probability  
    $\mathbb{P}(|a_{ij}|>Q_n)$. More precisely, we independently sample a random Bernoulli variable taking value 1 with possibility $\mathbb{P}(|a_{ij}|<Q_n)$ and taking value 0 otherwise, then we assign label $S$ if the Bernoulli variable equals 1 and assigns label $L$ if it equals 0.
    
\item Then for each $1\leq i\leq j\leq n$ such that $(i,j)$ is assigned with a label $L$, we independently sample $\widehat{a}_{ij}$ from the law $\mathbb{P}(a_{ij}\in dx\mid |a_{ij}|>Q_n)$ (that is, the law of $a_{ij}$ conditioned on the event that $|a_{ij}|>Q_n$) and then set $\widehat{a}_{ij}=\widehat{a}_{ji}$; and for $(i,j)$ assigned with label $S$ or for $(i,j)$ such that $b_{ij}=0$, we simply set $\widehat{a}_{ij}=0.$
\item From this, define for $(i,j)\in[1,n]^2:$ $$(\mathcal{M}_n^{L})_{ij}:=\frac{1}{\sqrt{k_n}}b_{ij}\widehat{a}_{ij},$$ and for each $c>0$ make the decomposition $\mathcal{M}_n^{L}=\mathcal{M}_n^{L,\leq c}+\mathcal{M}_n^{L,\geq c}$, where $\mathcal{M}_n^{L,\geq c}$ contains the elements in $\mathcal{M}_n^L$ having absolute value larger than $c$.
\item For each $1\leq i\leq j\leq n$ let $\overline{a}_{ij}$ be i.i.d. random variables (independent of the $\widehat{a}_{ij}$'s) having distribution $\mathbb{P}(a_{ij}\in dx\mid |a_{ij}|<Q_n)$ (that is, the law of $a_{ij}$ conditioned on the event that $|a_{ij}|<Q_n$), then set $\overline{a}_{ij}=\overline{a}_{ji}$ and define
$$(\mathcal{M}_n^{S})_{ij}:=\frac{1}{\sqrt{k_n}}b_{ij}\overline{a}_{ij}.$$

\item Finally,  set $$(\mathcal{M}_n^{N})_{ij}:=-\frac{1}{\sqrt{k_n}}b_{ij}\overline{a}_{ij}1_{\text{label of $(i,j)$ is $L$}},$$ then we have the equivalence in distribution
$$\mathcal{M}_n\overset{\text{law}}{\underset{}=}\mathcal{M}_n^S+\mathcal{M}_n^N+\mathcal{M}_n^{L,\leq c}+\mathcal{M}_n^{L,\geq c}.$$

\end{enumerate}    
\end{Sampling}

The following lemma can be proved in exactly the same way as Lemma \ref{lemma2.345} and \ref{lemma2.34567}, so the details are omitted.
\begin{lemma}\label{leaagagag} In the setting of Theorem \ref{theorem1.7}, the following estimates hold:
 \begin{enumerate}
        \item (Diagonal elements) We have 
        \begin{equation}
            \mathbb{P}\left(\text{ there exists $i\in[1,n]$ such that } |(\mathcal{M}_n)_{ii}|>\frac{1}{(\log n)^5}\right)=o(1).
        \end{equation}
        \item (Off-diagonal elements) We have 
         \begin{equation}
            \mathbb{P}\left(
        \begin{aligned}&\text{ there exists some $i\in[1,n]$ such that for two $j\neq k$, }\\& |(\mathcal{M}_n)_{ij}|>\frac{1}{(\log n)^5},\quad |(\mathcal{M}_n)_{ik}|>\frac{1}{(\log n)^5}\end{aligned}\right)=o(1).
        \end{equation}
        \item (Largest value) For any $\delta>0$, 
          \begin{equation}
         \lim_{n\to\infty} \lim_{M\to\infty}
            \mathbb{P}\left(\text{ there are $M$ subscripts $(i_s,j_s)$ such that } |(\mathcal{M}_n)_{i_sj_s}|>\delta\right)=0.
        \end{equation}
    \end{enumerate}
    Moreover, let $T_1\geq T_2\geq \cdots$ denote the set $\{|(\mathcal{M}_n)_{ij}|,1\leq i\leq j\leq n\}$ with elements re-arranged in decreasing order. Then in the case $k_n/n^\mu\to 1$, we have for any $x>0$,
$$
\mathbb{P}(T_1\leq x)=\exp\left(-\frac{cx^{-2(1+\frac{1}{\mu})}}{2}\right),
$$ and for any $c>0$, the random point process $\{|(\mathcal{M}_n)_{ij}|,1\leq i\leq j\leq n\}\cap [c,\infty)$  converges to the Poisson point process on $[c,\infty)$ with intensity measure $\frac{c(1+\frac{1}{\mu})}{x^{2(1+\frac{1}{\mu})+1}}dx.$
\end{lemma}

Let $\mathcal{F}_{S,L}$ denote the filtration generated by sampling the labels $S$ and $L$. Then thanks to this lemma and the definition of $Q_n$, there exists a Borel subset $\Omega_n\subset\Omega$ measurable with respect to $\mathcal{F}_{S,L}$ such that $\mathbb{P}(\Omega_n)\to 1$ and that on $\Omega_n$ we deterministically have $\|\mathcal{M}_n^N\|\leq (\log n)^{-5}$ and $\|\mathcal{M}_n^{L,\leq c}\|\leq c$. Moreover, on $\Omega_n$, nonzero eigenvalues of $\mathcal{M}_n^{L,\geq c}$ exactly correspond to the nonzero elements of $\mathcal{M}_n^{L,\geq c}$ times $\pm 1.$

\subsection{Finite rank perturbation for weighted regular graphs} We will use as a key technical input the finite rank perturbation result from Benson Au \cite{au2023bbp}.

\begin{theorem}\label{theorem3.456}(\cite{au2023bbp}, Theorem 1.3) For each $1\leq i\leq j\leq n$ assume the random variables $\mathfrak{a}_{ij}$ are i.i.d., $\mathbb{E}[\mathfrak{a}_{ij}]=0$, $\mathbb{E}[|\mathfrak{a}_{ij}|^2]=1$ and for each $p\in\mathbb{N}_+$, 
$\mathbb{E}[|\mathfrak{a}_{ij}|^p]<\infty$. Set $\mathfrak{a}_{ji}=\mathfrak{a}_{ij}.$ Let the matrix $G_n=(g_{ij})$ be given as in Theorem \ref{theorem1.7} and define the following matrix
 $$(\mathfrak{M}_n)_{ij}:=\frac{1}{\sqrt{k_n}}g_{ij}\mathfrak{a}_{ij},\quad (i,j)\in[1,n]^2.$$

 Let $A_n$ be a $n\times n$ matrix with finite rank, whose rank and eigenvalues do not depend on $n$. Let $\theta_1\geq\theta_2\geq\cdots \geq\theta_r\geq0\geq\theta_{r+1}\geq\cdots\geq\theta_s$ denote the eigenvalues of $A_n$, then:
\begin{enumerate}
    \item For $1\leq i\leq r$, $\lambda_i(A_n+\mathfrak{M}_n)\overset{\text{law}}{\underset{n\to\infty}\longrightarrow} f(\theta_i),$ and $\lambda_i(A_n+\mathfrak{M}_n)\overset{\text{law}}{\underset{n\to\infty}\longrightarrow} 2$ for $i\geq r+1,$
    \item For each $1\leq i,j\leq r$ such that $\theta_j\neq \theta_i$ and $\theta_i>1$, 
\begin{equation}
    \|P_i(A_n)v_i(A_n+\mathfrak{M}_n)\|^2\overset{\text{law}}{\underset{n\to\infty}\longrightarrow} 1-\frac{1}{\theta_i^2},\quad  \|P_j(A_n)v_i(A_n+\mathfrak{M}_m)\|^2\overset{\text{law}}{\underset{n\to\infty}\longrightarrow} 0,
\end{equation} where $v_i(A_n+\mathfrak{M}_n)$ denotes the eigenvector associated to the $i$-th largest eigenvalue of $A_n+\mathfrak{M}_n$ and $P_i$ is projection to the eigenspace of the $i$-th eigenvalue of $A_n$.
\end{enumerate}

\end{theorem}

We show the moment condition of this theorem can be singnificantly weakened:

\begin{Proposition}\label{finiterankmodified}
    In the setting of Theorem \ref{theorem1.7}, all the claims in Theorem \ref{theorem3.456} remain true if we replace $\mathfrak{M}_n$ by $\mathcal{M}_n^S$.
\end{Proposition}

\begin{proof} The proof will again proceed through applying the universality principle in \cite{brailovskaya2022universality}, Section 3.5. We first choose $a_n\to 1$ so that $a_n\mathcal{M}_n^S$ has the same mean and variance profile as $\mathfrak{M}_n$. Observe that by construction, $R(\mathcal{M}_n^S)\leq (\log n)^{-5}$. Then rewriting the proof of \cite{brailovskaya2022universality}, Theorem 3.31 with $\mathfrak{M}_n$ in place of the GOE matrix, we can prove that all the claims of Theorem \ref{theorem3.456} apply also to $a_n\mathcal{M}_n^S$. Indeed, the only property of GOE used in the proof of that theorem is that standard GOE matrix satisfies all the claims (1),(2) of Theorem \ref{theorem3.456}, in addition to having the standard variance profile and having all moments finite; and all of these properties have been shown to hold for $\mathfrak{M}_n$ as well by Theorem \ref{theorem3.456}. Thus having proved this proposition for $a_n\mathcal{M}_n^S$, it suffices to note that $\|(1-a_n)\mathcal{M}_n^S\|\to 0$, where finiteness whp. of $\|\mathcal{M}_n^S\|$ again follow from \cite{brailovskaya2022universality}, Corollary 2.6. Finally, we use Lemma \ref{lemmaline422} to remove the small norm matrix and obtain eigenvector statistics. (Note: if one only wishes to derive the distribution of the largest eigenvalue, then applying \cite{brailovskaya2022universality}, Corollary 2.6 is enough, and we do not need to rewrite the proof of \cite{brailovskaya2022universality}, Theorem 3.31.)

\end{proof}

\subsection{Finishing the proof}
\begin{proof}[\proofname\ of Theorem \ref{theorem1.7}]
    The proof follows from a combination of sampling algorithm in Sampling \ref{sampling3.1}, the distribution estimate \eqref{leaagagag}, and the finite rank perturbation result in Proposition \eqref{finiterankmodified}. The details are exactly the same as prior cases and are thus omitted.
\end{proof}

\section{Proof for sample covariance matrices}\label{1section4}

In this section we prove Theorem \ref{theorem1.9}, the top eigenvalue of sample covariance matrix. For sample covariance matrices we have a very simple reduction to the Wigner case: the top eigenvalue of $\frac{1}{L}(S_{L,M})(S_{L,M})^*$ is the square of the top eigenvalue of
$$ E_{L,M}:=
\begin{pmatrix}
0&\frac{1}{\sqrt{L}}S_{L,M}\\\frac{1}{\sqrt{L}}(S_{L,M})^*&0.
\end{pmatrix}
$$
Thus we will prove Theorem \ref{theorem1.9} via three steps: (1) resample $S_{L,M}$ and express it as a small rank perturbation of another i.i.d. matrix with small entries; (2) prove some finite rank perturbation results for $E_{L,M}$ when $\frac{1}{\sqrt{L}}S_{L,M}$ is replaced by some $\frac{1}{\sqrt{L}}S_{L,M}+A_M$, and that $S_{L,M}$ has entries with all moments finite; and (3) extend that perturbation result to only requiring each entry of $\frac{1}{\sqrt{L}}S_{L,M}$ are small, and then conclude.

\subsection{Sampling procedure}

\begin{Sampling}\label{sampling4.1} In the setting of Theorem \ref{theorem1.9}, we take the following steps to resample $S_{L,M}$. Define $Q_N:=\sqrt{N}(\log N)^{-5}$,
\begin{enumerate}
\item First, for each $1\leq i\leq L$, $1\leq j\leq M$ we independently place a label $S$ on $(i,j)$ with probability $\mathbb{P}(|(S_{L,M})_{ij}|<Q_N)$, and place a label $L$ with probability  
    $\mathbb{P}(|(S_{L,M})_{ij}|>Q_N)$. More precisely, for each $1\leq i\leq L,1\leq j\leq M$ we independently sample a Bernoulli random variable which takes value 1 with probability $\mathbb{P}(|(S_{L,M})_{ij}|<Q_N)$ and 0 otherwise. Then assign a label $S$ to site $(i,j)$ if the Bernoulli variable takes value 1 and assign a label $L$ if the Bernoulli variable takes value 0.
    Contrary to previous examples, here we consider directed edges, which means that the label on $(i,j)$ is independent of the label on  $(j,i)$ whenever $i\neq j$.
    
\item Then for each $1\leq i\leq L$, $1\leq j\leq M$ such that $(i,j)$ is assigned with a label $L$, we independently sample $\widehat{S}_{ij}$ from the law $\mathbb{P}((S_{L,M})_{ij}\in dx\mid |(S_{L,M})_{ij}|>Q_N)$ (that is, the law of $(S_{L,M})_{ij}$ conditioned on the event that $|(S_{L,M})_{ij}|>Q_N$); and for $(i,j)$ assigned with label $S$, we simply set $\widehat{S}_{ij}=0.$
\item From this, define for $1\leq i\leq L,1\leq j\leq M$, $$(\mathcal{S}_{L,M}^{L})_{ij}:=\frac{1}{\sqrt{N}}\widehat{S}_{ij},$$ and for each $c>0$ make the decomposition $\mathcal{S}_{L,M}^{L}=\mathcal{S}_{L,M}^{L,\leq c}+\mathcal{S}_{L,M}^{L,\geq c}$, where $\mathcal{S}_{L,M}^{L,\geq c}$ contains the elements in $\mathcal{S}_{L,M}^L$ having absolute value larger than $c$.
\item For each $1\leq i\leq L$,$1\leq j\leq M$, let $\overline{S}_{ij}$ be i.i.d. random variables (independent of the $\widehat{a}_{ij}$'s) having distribution $\mathbb{P}((S_{L,M})_{ij}\in dx\mid |(S_{L,M})_{ij}|<Q_N)$ (that is, the law of $(S_{L,M})_{ij}$ conditioned on the event that $|(S_{L,M})_{ij}|<Q_N$), then define
$$(\mathcal{S}_{L,M}^{S})_{ij}:=\frac{1}{\sqrt{N}}\overline{S}_{ij}.$$

\item Finally,  set $$(\mathcal{S}_{L,M}^{N})_{ij}:=-\frac{1}{\sqrt{N}}\overline{S}_{ij}1_{\text{label of $(i,j)$ is $L$}},$$ then via a simple computation of conditional distributions, we have the equivalence in distribution
\begin{equation}\label{indistributionmode}
\mathcal{S}_{L,M}\overset{\text{law}}{\underset{}=}\sqrt{\frac{N}{L}}\left(\mathcal{S}_{L,M}^S+\mathcal{S}_{L,M}^N+\mathcal{S}_{L,M}^{L,\leq c}+\mathcal{S}_{L,M}^{L,\geq c}\right),\end{equation} where the factor $\sqrt{N/L}$ comes from the fact that we took a different scaling.

Just as in Lemma \ref{leaagagag}, we can prove the following probabilistic characterization:

\begin{lemma} In the setting of Theorem \ref{theorem1.9}, the following claims hold with probability tending to one:
\begin{enumerate}
    \item For each $1\leq i\leq L$ there exists at most one $j\in[1,M]$ such that $(S_{L,M}^L)_{ij}\neq 0$. 
    \item For each $1\leq j\leq M$ there exists at most one $i\in[1,L]$ such that $(S_{L,M}^L)_{ij}\neq 0$. 
\end{enumerate} Further, for any $c>0$,
$\lim_{D\to\infty}\mathbb{P}(\operatorname{Rank}(S_{L,M}^{L,\geq c})\geq D)= 0$ for $N$ sufficiently large.

Let $T_1$ denote the largest element among $\{\sqrt{\frac{1}{N}}\left|(S_{L,M})_{i,j}\right|,1\leq i\leq L,1\leq j\leq M\}$. Then for any $x>0$,
\begin{equation}
    \mathbb{P}(T_1\geq x)=e^{-\frac{c\alpha x^{-4}}{(1+\alpha)^2}}.
\end{equation}

\end{lemma}

\begin{proof} The first three claims can be proved in exactly the same way as Lemma \ref{leaagagag}, hence omitted. We proceed to characterize the distribution of $T_1$: by independence,
\begin{equation}
    \mathbb{P}(T_1\geq x)=\prod_{i,j}\mathbb{P}(|(S_{L,M})_{ij}|\geq x\sqrt{N})=(1-c\frac{1}{N^{2}x^{4}})^{LM}{\underset{N\to\infty}\longrightarrow} e^{-\frac{c\alpha x^{-4}}{(1+\alpha)^2}}.
\end{equation}

\end{proof}

\end{enumerate}    
\end{Sampling}

\subsection{Finite rank perturbations}
We begin with the following technical computation.
\begin{Proposition}\label{proposition4.3} Let $H_N=(h_{ij})$ be a $L\times M$ matrix with i.i.d. elements, such that $\mathbb{E}[h_{ij}]=0,$ $\mathbb{E}[|h_{ij}|^2]=1$, and for each $p>1$, $\mathbb{E}[|h_{ij}|^p]<\infty.$ Let $A_N$ be a $L\times M$ nonrandom matrix with rank $O(1)$ independent of $M$, and that in any row or column of $A_N$ there is at most one nonzero element. Assume that the multi-set $\{|(A_N)_{ij}|,1\leq i\leq L,i\leq j\leq M\}$ is independent of $M$, and denote by $x:=\max \{|(A_N)_{ij}|,1\leq i\leq L,1\leq j\leq M\}$. Consider 
$$
K_N^A:=\begin{pmatrix}  0& \frac{1}{\sqrt{N}}H_N+A_N\\\frac{1}{\sqrt{N}}(H_N)^*+(A_N)^*&0\end{pmatrix},
$$ then as $N\to\infty$, $\lambda_1(K_N^A)$ converges in distribution to $F_\alpha(x)$, where $F_\alpha$ is a deterministic function defined in \eqref{functionoffalpha} that depends only on $\alpha$.
\end{Proposition}

\begin{proof} We follow some computations in \cite{WOS:000542157900013}, Section 5.2. Let $K_N$ denote the matrix resulting from $K_N^A$ by setting $A_N=0$. Then $z\in\mathbb{R}$ is an eigenvalue of $K_N^A$ if and only if $\det(z\operatorname{I}_N-K_N^A)=0$. If $z$ lies outside the spectrum of $K_N$ (note that the top eigenvalue of $K_N$ is no larger than the square root of the top eigenvalue of $\frac{1}{N}H_N(H_N)^*$, so this holds for any $z>\frac{1+\sqrt{\alpha}}{\sqrt{1+\alpha}}$), then this implies 

\begin{equation}\label{equationfordeterminant}\det\left(\operatorname{I}_N-(z\operatorname{I}_N-K_N)^{-1}\begin{pmatrix}0&A_N\\A_N^*&0\end{pmatrix}\right)=0.\end{equation}

We rewrite $$(z\operatorname{I}_N-K_N)^{-1}=\begin{pmatrix} R_{11}(z)&R_{12}(z)\\R_{21}(z)&R_{22}(z)\end{pmatrix},$$
then the condition on $z$ can be restated as 
$$
\det\left(\operatorname{I}_N-\begin{pmatrix}  R_{12}(z)A_N^*&  R_{11}(z)A_N\\  R_{22}(z)A_N^*& R_{21}(z)A_N \end{pmatrix}\right)=0.
$$

By \cite{WOS:000542157900013}, Lemma 5.7, we have that for any $\delta>0$,
$$
\sup_{\substack{\|v\|=\|w\|=1,\\v\in\mathbb{R}^M,w\in\mathbb{R}^L}}\mathbb{P}\left(\sup_{z\geq \epsilon+\frac{1+\sqrt{\alpha}}{\sqrt{1+\alpha}}} \left|\langle v,R_{2,1}(z)w\rangle\right|>\delta
\right)\to 0,$$
and that from \cite{WOS:000542157900013}, Lemma 5.7 and using the isotropic local law from \cite{articlelocallaw}, Theorem 2.5:
$$
\sup_{\substack{\|v\|=\|w\|=1,\\v,w\in\mathbb{R}^L}}\mathbb{P}\left(\sup_{z\geq \epsilon+\frac{1+\sqrt{\alpha}}{\sqrt{1+\alpha}}} \left|\langle v,R_{1,1}(z)w\rangle-z(1+\alpha)\langle v,w\rangle G_{MP(\alpha)}((1+\alpha)z^2)\right|>\delta
\right)\to 0,$$

$$
\sup_{\substack{\|v\|=\|w\|=1,\\v,w\in\mathbb{R}^M}}\mathbb{P}\left(\sup_{z\geq \epsilon+\frac{1+\sqrt{\alpha}}{\sqrt{1+\alpha}}} \left|\langle v,R_{2,2}(z)w\rangle-z(1+\alpha)\langle v,w\rangle G_{MP(\alpha^{-1})}((1+\alpha)z^2)\right|>\delta
\right)\to 0,$$
where $G_{MP(\alpha)}(z)$ is the Stieltjes transform for the $\alpha$- Marchenko–Pastur law \eqref{pasteurlaw}.

Taking these computations back to \eqref{equationfordeterminant}, we see that as $N\to\infty$, a solution $z>\frac{1+\sqrt{a}}{\sqrt{1+a}}$ to \eqref{equationfordeterminant} must asymptotically satisfy
\begin{equation}
    z(1+\alpha)G_{MP(\alpha)}((1+\alpha)z^2)\times \left(z(1+\alpha)G_{MP(\alpha^{-1})}((1+\alpha)z^2)\right)=\frac{1}{a_{jk}^2}
\end{equation}
for some nonzero element $a_{jk}$ appearing in the matrix $A_N$.

This equation is hard to solve explicitly for general $\alpha$, but we have the following properties: both $z\to zG_{MP(\alpha)}((1+\alpha)z^2)$ and $z\to zG_{MP(\alpha^{-1})}((1+\alpha)z^2)$ are positive and monotonically decreasing on $z\in(\frac{1+\sqrt{\alpha}}{\sqrt{1+\alpha}},\infty),$ so that there exists a unique $\tau_\alpha>0$ such that for any $x>\tau_\alpha$, there is a unique $z\in (\frac{1+\sqrt{\alpha}}{\sqrt{1+\alpha}},\infty)$ that solves 
\begin{equation}\label{equationx}
    z^2(1+\alpha)^2G_{MP(\alpha)}((1+\alpha)z^2)G_{MP(\alpha^{-1})}((1+\alpha)z^2)=\frac{1}{x^2},
\end{equation}
and we define $F_\alpha(x)$ for any $x>0$ as follows:
\begin{equation}\label{functionoffalpha}
    F_\alpha(x):=\begin{cases}  \text{ the unique solution $z>\frac{1+\sqrt{\alpha}}{\sqrt{1+\alpha}}$ to \eqref{equationx}},\quad x>\tau_\alpha,\\
    \frac{1+\sqrt{\alpha}}{\sqrt{1+\alpha}},  \quad x<\tau_\alpha.
    \end{cases}
\end{equation}
Clearly we have
\begin{equation}
    \tau_\alpha=\frac{1}{\sqrt{(1+\alpha)(1+\sqrt{\alpha})^2G_{MP(\alpha)}((1+\sqrt{\alpha})^2)G_{MP(\alpha^{-1})}((1+\sqrt{\alpha})^2)}}
\end{equation}
This completes the proof.
\end{proof}

\begin{remark}
    As a reality check, we may take $\alpha=1$, and we have that 
    $$
2zG_{MP(1)}(2z^2)=z-\sqrt{z^2-2},
    $$ and a solution to $z-\sqrt{z^2-2}=\frac{1}{a}$ is given by  $z=a+\frac{1}{2a}$. In this choice $\alpha=1$, the matrix $K_N$ is indeed the weighted adjacency matrix of a $M$-regular graph, while we have divided by $\sqrt{N}$ rather than $\sqrt{M}$, which differ by a multiple of $\sqrt{2}$. The solution $z=a+\frac{1}{2a}$ exactly gives the finite rank perturbation result in Proposition \ref{theorem3.456} (modulo the fact that in our scaling the variance of each matrix entry is multiplied by $\frac{1}{2}$). Meanwhile, $G_{MP(1)}(4)=\frac{1}{2},$ so that $\tau_\alpha=\frac{1}{\sqrt{2}}$,  also consistent with Proposition \ref{theorem3.456}.
\end{remark}

Now we show the moment conditions in Proposition \ref{proposition4.3} can be weakened.

\begin{Proposition}\label{proposition4.5} In the setting of Theorem \ref{theorem1.9}, consider the matrix $S_{L,M}^S$ defined in the sampling procedure \ref{sampling4.1}. Then all the claims of Proposition \ref{proposition4.3} remain true if we replace $\frac{1}{\sqrt{N}}H_N$ by $S_{L,M}^S$.
\end{Proposition}

\begin{proof}
We can find a sequence of constants $a_N\to 1$ as $N\to\infty$ such that $a_NS_{L,M}^S$ has the same mean and variance profile as the matrix $H_N$ in Proposition \ref{proposition4.3}. Moreover, by construction $R(S_{L,M}^S)\leq (\log N)^{-5}$, hence we can apply the universality principle in \cite{brailovskaya2022universality}, Theorem 2.7 to conclude that both $
\begin{pmatrix}  0& \frac{1}{\sqrt{N}}H_N\\\frac{1}{\sqrt{N}}(H_N)^*&0\end{pmatrix}
$ and $
\begin{pmatrix}  0& a_NS_{L,M}^S\\(a_NS_{L,M}^S)^*&0\end{pmatrix},
$ have approximately the same top eigenvalue when perturbed by $\begin{pmatrix}  0& A_N\\(A_N)^*&0\end{pmatrix}.
$ The $(1-a_N)S_{L,M}^S$ term can be removed by Weyl interlacing formula and the fact that $\|S_{L,M}^S\|$ is bounded with high probability, which is also a consequence of \cite{brailovskaya2022universality}, Corollary 2.6 and Remark 2.1 (extension to non self-adjoint matrices).
\end{proof}

\subsection{Completing the proof} We can now complete the proof of Theorem \ref{theorem1.9}.

\begin{proof}[\proofname\ of Theorem \ref{theorem1.9}] Through the sampling procedure \ref{sampling4.1}, we first sample the labels $S$, $L$ and sample entries in $S_{L,M}^L$. The distribution of  $S_{L,M}^S$ is not changed under this sampling, and on an event $\Omega_n\subset\Omega$ measurable with respect to sampling, with $\mathbb{P}(\Omega_n)\to 1$, we have $\|S_{L,M}^N\|\leq (\log N)^{-5}$ and is thus negligible. Then we apply Proposition \ref{proposition4.5} with $S_{L,M}^{L,\geq c}$ in place of $A_N$ (knowing that the largest absolute value of elements in $S_{L,M}^{L,\geq c}$ converges to $\xi_{c,\alpha}$, we only need to add a matrix $A_N^1$ with vanishing  operator norm so that  $S_{L,M}^{L,\geq c}+A_N^1$ has $N$-independent nonzero elements, and the addition of $A_N^1$ can be ignored by Weyl interlacing inequality). Noting that 
    $\|S_{L,M}^N\|\to 0$, that the limit is independent of $c>0$ small so we can set $c\to 0$, and that we have the decomposition \eqref{indistributionmode}, now the proof of Theorem \ref{theorem1.9} is complete by the first paragraph in the beginning of Section \ref{1section4}.
\end{proof}

\printbibliography

\end{document}